\documentclass{article}

\usepackage{graphicx}
\usepackage{amsmath}
\usepackage{amsthm}
\usepackage{amstext}
\usepackage{amsfonts}
\usepackage{amssymb}
\usepackage{amsxtra}
\usepackage{bbm}
\usepackage{url}
\usepackage{tikz}
\usetikzlibrary{external}
\usepackage{pgfplots}
\usepackage{pgfplotstable}
\pgfplotsset{compat=newest}

\usepackage[square,numbers,sort&compress]{natbib}
\usepackage{algorithm}
\usepackage{algorithmic}

\newtheorem{thm}{Theorem}[section]

\newtheorem{lem}[thm]{Lemma}
\newtheorem{cor}[thm]{Corollary}
\theoremstyle{definition}

\newtheorem{definition}[thm]{Definition}
\newtheorem{exmp}[thm]{Example}

\newtheorem*{cexmp*}{Counterexample}
\newtheorem{rem}[thm]{Remark}

\newcommand \be {\begin{equation}}
\newcommand \ee {\end{equation}}
\newcommand \ben {\begin{equation*}}
\newcommand \een {\end{equation*}}
\newcommand \bea {\begin{eqnarray}}
\newcommand \eea {\end{eqnarray}}
\newcommand \bean {\begin{eqnarray*}}
\newcommand \eean {\end{eqnarray*}}

\newcommand \NN {\mathbb{N}}
\newcommand \RR {\mathbb{R}}

\newcommand \EE {\mathbb{E}}

\newcommand \Rcal {\mathcal{R}}

\newcommand \Ncal {\mathcal{N}}

\DeclareMathOperator*{\argmin}{argmin}
\DeclareMathOperator{\sign}{sign}

\DeclareMathOperator{\rint}{rint}

\DeclareMathOperator{\conv}{conv}
\DeclareMathOperator{\dist}{dist}

\DeclareMathOperator{\trace}{trace}

\newcommand{\scp}[2]{\langle #1 \,,\, #2\rangle}
\newcommand{\norm}[2][]{\|#2\|_{#1}}
\newcommand{\set}[2]{\{#1\, |\, #2\}}

\begin{document}

\title{Linear convergence of the Randomized Sparse Kaczmarz Method}

\author{Frank~Sch\"{o}pfer\thanks{Institut f\"ur Mathematik, Carl von Ossietzky Universit\"at Oldenburg, 26111 Oldenburg, Germany, \texttt{frank.schoepfer@uni-oldenburg.de}} \and Dirk~A. Lorenz\thanks{Institute for Analysis and Algebra, TU Braunschweig, 38092 Braunschweig, Germany, \texttt{d.lorenz@tu-braunschweig.de}, fon +49-531-391-7423, fax +49-531-391-7414. The work of D.L. was partially supported by the National Science Foundation under Grant DMS-1127914 to the Statistical and Applied Mathematical Sciences Institute. Any opinions, findings, and conclusions or recommendations expressed in this material are those of the author(s) and do not necessarily reflect the views of the National Science Foundation.}}

\maketitle

\begin{abstract}
  The randomized version of the Kaczmarz method for the solution of
  linear systems is known to converge linearly in expectation.  In
  this work we extend this result and show that the recently proposed
  Randomized Sparse Kaczmarz method for recovery of sparse solutions,
  as well as many variants, also converges linearly in
  expectation. The result is achieved in the framework of split feasibility problems and their solution by randomized
  Bregman projections with respect to strongly convex functions. To obtain the expected convergence rates we prove extensions
  of error bounds for projections. The convergence result is shown to hold in more
  general settings involving smooth convex functions, piecewise
  linear-quadratic functions and also the regularized nuclear norm,
  which is used in the area of low rank matrix problems. Numerical
  experiments indicate that the Randomized Sparse Kaczmarz method
  provides advantages over both the non-randomized and the non-sparse
  Kaczmarz methods for the solution of over- and under-determined
  linear systems.
\end{abstract}



\section{Introduction}

In this paper we analyse a randomized variant of the recently proposed
Sparse Kaczmarz method to recover sparse solutions of linear systems. Let
$A\in\RR^{m\times n}$ be a matrix with rows $a_{i}^{T}\in \RR^{n}$ and
$b\in \RR^{m}$ be such that the linear system $Ax=b$ is
consistent. For the standard Kaczmarz method~\cite{Kac37} one goes
through the indices of the rows cyclically, and projects a given iterate
onto the solution space of this row. For $i= \mathrm{mod}(k-1,m)+1$ the method iterates
\begin{equation}
  \label{eq:kaczmarz-iter}
  x_{k+1} = x_{k}-\frac{\scp{a_{i}}{x_{k}} - b_{i}}{\norm[2]{a_{i}}^{2}} \cdot a_{i}.
\end{equation}
It is known that the method converges to the minimum norm solution $\hat{x}$
of $Ax=b$ when it is initialized with $x_{0}=0$, but the speed of convergence is not simple to quantify, and especially, depends on the
ordering of the rows, see e.g.~\cite{DH97}. The situation changes if one considers a
randomization such that in each step one chooses a row of the system
at random. In the seminal paper~\cite{SV09} it has been shown that a choice of row $i$ with probability
$\norm[2]{a_{i}}^{2}/\norm[F]{A}^{2}$ leads to a linear convergence rate in expectation,
\[
\EE\left[\norm[2]{x_{k+1}-\hat{x}}^2 \right] \le (1- \tfrac{\sigma_{\min}^{2}}{\norm[F]{A}^{2}}) \cdot \EE\left[\norm[2]{x_{k}-\hat{x}}^2\right] \,,
\]
where $\norm[F]{A}^{2}$ is the Frobenius norm and $\sigma_{\min}$ denotes the smallest positive singular value of $A$.
Since then similar results have been obtained for randomized Block Kaczmarz methods and systems of equalities and inequalities, see~\cite{LL10,BN14,NT14} and connections to stochastic gradient descent have been drawn~\cite{NSW16}.

In~\cite{LSW14,LWSM14} a variant of the Kaczmarz method has been
proposed that produces sparse solutions. This \emph{Sparse Kaczmarz method} uses two variables and reads as
\begin{equation}
  \label{eq:sparse-kaczmarz-iter}
  \begin{split}
    x^*_{k+1} & = x^*_{k}-\frac{\scp{a_{i}}{x_{k}} - b_{i}}{\norm[2]{a_{i}}^{2}} \cdot a_{i}\\
    x_{k+1} & = S_{\lambda}(x^*_{k+1})
  \end{split}
\end{equation}
with $\lambda>0$ and the soft shrinkage function
$S_{\lambda}(x) = \max\{|x|-\lambda,0\} \cdot \sign(x)$.
It has been shown in~\cite{LSW14} that the iterates $x_k$ converge to the solution of the \emph{regularized Basis Pursuit problem},
\be
\min_{x \in \RR^n} \lambda\norm[1]{x} + \tfrac{1}{2}\norm[2]{x}^{2}  \quad{s.t.}\quad  Ax=b\,, \label{eq:BP}
\ee
see e.g.~\cite{COS09,CDS98,Ela10}, and also~\cite{Sch12} for explicit values of $\lambda>0$ that guarantee exact recovery of sparse solutions.
But no convergence rate has been given.
In~\cite{Pet15} sublinear convergence rates have been obtained for the \emph{Randomized Sparse Kaczmarz method} by identifying the iteration as a randomized coordinate gradient descent method applied to the unconstrained dual of~\eqref{eq:BP}, see also~\cite{Nes12,WB16}.
However, linear convergence could only be obtained by smoothing the objective function in~\eqref{eq:BP}, which results in an iteration that is slightly different from~\eqref{eq:sparse-kaczmarz-iter}, and need not solve~\eqref{eq:BP}.
Here we will show that the Randomized Sparse Kaczmarz method in fact converges linearly in expectation without smoothing.
We use the theoretical framework
developed in~\cite{LSW14}, which treats the Sparse Kaczmarz method as a
special case of so-called Bregman projections for split
feasibility problems. Using this flexible framework we will show
(sub-)linear convergence rates for a broad range of problems.
Especially, linear rates are also obtained for randomized iterations of the form
\begin{equation}
  \label{eq:nuclear-iter}
  \begin{split}
    X^*_{k+1} & = X^*_{k}-\tfrac{\scp{A_{i}}{X_{k}} - b_{i}}{\norm[F]{A_{i}}^{2}} \cdot A_{i}\\
    X_{k+1} & = S_{\lambda}(X^*_{k+1})
  \end{split}
\end{equation}
to solve the \emph{regularized nuclear norm} optimization problem in the area of low rank matrix problems,
\be
\min_{X \in \RR^{n_1 \times n_2}} \lambda\norm[*]{X} + \tfrac{1}{2}\norm[F]{X}^{2}  \quad{s.t.}\quad  \scp{A_{i}}{X} = b_{i}\:,\: i=1,\ldots,m \,,\label{eq:N}
\ee
where $\scp{A}{X}=\trace(A^T \cdot X)$ for two matrices $A, X \in \RR^{n_1 \times n_2}$, and $S_{\lambda}(X)$ denotes the singular value thresholding operator, see eg.~\cite{LY13,RFP10, ZCCZ12, CCS10}.

In the next section we recall the basic properties of Bregman projections.
In section 3 we prove some error bounds which are crucial for the convergence analysis of the method of randomized Bregman projections in section 4.
The special case of the Randomized Sparse Kaczmarz method is treated in section 5.
In the last section we report some numerical results illustrating the performance of the Sparse Kaczmarz method with and without randomization, and also its benefit for sparsity problems compared to the standard Kaczmarz method, even in the case of overdetermined systems.

\section{Basic notions}

We recall some well known concepts and properties of convex functions,
see~\cite{RW09}, and state basic assumption that will be used throughout
the paper.

Let $f:\RR^n \to \RR$ be convex.
Since $f$ is assumed
to be finite everywhere, it is also continuous.
By $\partial f(x)$
we denote the subdifferential of $f$ at $x \in \RR^n$,
\ben
\partial f(x) = \set{x^* \in \RR^n}{ f(y) \ge f(x) + \scp{x^*}{y-x}\quad \mbox{for all $y \in \RR^n$}}\,,
\een
which is nonempty, compact and convex.
Furthermore for all $R>0$ we have
\[
\sup_{x \in B_R \,,\, x^* \in \partial f(x)} \norm[2]{x^*}  < \infty \quad,\quad\mbox{where} \quad B_R:=\set{x \in \RR^n}{\norm[2]{x}\le R} \,.
\]
\begin{definition}
The convex function $f:\RR^n \to \RR$ is said to be \emph{$\alpha$-strongly convex} for some $\alpha>0$, if for all $x,y \in \RR^n$ and $x^* \in \partial f(x)$ we have
\ben
f(y) \ge f(x) + \scp{x^*}{y-x} + \frac{\alpha}{2} \cdot \norm[2]{y-x}^2 \,.
\een
\end{definition}

The convex conjugate function of $f$ is $f^*:\RR^n \to \RR$,
\ben
f^*(x^*)= \sup_{x \in \RR^n} \scp{x^*}{x} - f(x) \,.
\een

\begin{thm}
If $f:\RR^n \to \RR$ is $\alpha$-strongly convex then the conjugate function $f^*$ is differentiable with a $1/\alpha$-Lipschitz-continuous gradient, i.e.
\ben
\norm[2]{\nabla f^*(x^*)-\nabla f^*(y^*)} \le \frac{1}{\alpha} \cdot \norm[2]{x^*-y^*} \quad \mbox{for all $x^*,y^* \in \RR^n$.}
\een
\end{thm}

\begin{definition} \label{def:cpq}
A convex function $f:\RR^n \to \RR$ is called \emph{piecewise linear-quadratic} if there are finitely many polyhedral sets $F_i \subset \RR^n$, $i \in I:=\{1,\ldots,p\}$, whose union equals $\RR^n$, and relative to each of which $f(x)$ is given by a convex linear-quadratic function
\ben
f(x)=\tfrac{1}{2} \cdot \scp{x}{A_i x} +\scp{a_i}{x} + \alpha_i \quad,\quad x \in F_i \,,
\een
with symmetric positive-semidefinite matrices $A_i \in \RR^{n \times n}$, vectors $a_i \in \RR^n$ and $\alpha_i \in \RR$.
For $x \in \RR^n$ we define $I_f(x):=\set{i \in I}{x \in F_i}$ and $F_x:=\bigcap_{i \in I_f(x)} F_i$.
\end{definition}

Note that each $F_x$ is polyhedral and there are only finitely many different sets $F_x$.

\begin{thm} \label{thm:partialfcpq}
If $f:\RR^n \to \RR$ is convex piecewise linear-quadratic then $f^*$ is also convex piecewise linear-quadratic, and for all $x \in \RR^n$ we have
\ben
\partial f(x)=\conv\set{A_i x + a_i}{i \in I_f(x)} \,.
\een
\end{thm}

\subsection{Bregman distance}

The concept of Bregman distance and projections goes back to Bregman~\cite{Bre67} and has been successfully used in optimization, see e.g.~\cite{SSL08b,AB97,BB97,BC01,Bur16}.
The definitions and results in this and the next subsection are taken from~\cite{LSW14}.

\begin{definition}
Let $f:\RR^n \to \RR$ be strongly convex.
The \emph{Bregman distance} $D_f^{x^*}(x,y)$ between $x,y \in \RR^n$ with respect to $f$ and a subgradient $x^* \in \partial f(x)$ is defined as
\ben
D_f^{x^*}(x,y):=f(y)-f(x) -\scp{x^*}{y - x} = f^*(x^*)-\scp{x^*}{y} + f(y)\,.
\een
If $f$ is differentiable then we have $\partial f(x)=\{\nabla f(x)\}$ and hence we simply write $D_f(x,y)=D_f^{x^*}(x,y)$.
\end{definition}

Note that for $f(x)=\frac{1}{2}\norm[2]{x}^2$ we just have $D_f(x,y)=\frac{1}{2}\norm[2]{x-y}^2$.
In general $D_f$ is not a distance function in the usual sense, as it need neither be symmetric, nor does it have to obey a (quasi-)triangle inequality.
Nevertheless it has some distance-like properties which we state in the following lemma.

\begin{lem} \label{lem:D}
Let $f:\RR^n \to \RR$ be $\alpha$-strongly convex.
For all $x,y \in \RR^n$ and $x^* \in \partial f(x)$, $y^* \in \partial f(y)$ we have
\ben
\frac{\alpha}{2} \norm[2]{x-y}^2 \le  D_f^{x^*}(x,y) \le \scp{x^*-y^*}{x-y} \le \norm[2]{x^*-y^*} \cdot \norm[2]{x-y}
\een
and hence
\ben
D_f^{x^*}(x,y) = 0 \quad \Leftrightarrow \quad x=y.
\een
For sequences $x_k$ and $x_k^*\in \partial f(x_k)$ boundedness of $D_f^{x_k^*}(x_k,y)$ implies boundedness of both $x_k$ and $x_k^*$.
If $f$ has a $L$-Lipschitz-continuous gradient then we also have $D_f(x,y) \le \tfrac{L}{2} \cdot \norm[2]{x-y}^2$.
\end{lem}

\subsection{Bregman projections}

\begin{definition}
Let $f:\RR^n \to \RR$ be strongly convex, and $C \subset \RR^n$ be a nonempty closed convex set. 
The \emph{Bregman projection} of $x$ onto $C$ with respect to $f$ and $x^* \in \partial f(x)$ is the unique point $\Pi_C^{x^*}(x) \in C$ such that
\ben
D_f^{x^*}\big(x,\Pi_C^{x^*}(x)\big) = \min_{y \in C} D_f^{x^*}(x,y) =: \dist_f^{x^*}(x,C)^2\,.
\een
For differentiable $f$ we simply write $\Pi_C(x)$ and $\dist_f(x,C)$.
\end{definition}

The notation for the Bregman projection does not capture its dependence on the function $f$, which, however, will always be clear from the context.
Note that for $f(x)=\frac{1}{2}\|x\|_2^2$ the Bregman projection is just the orthogonal projection onto $C$.
To distinguish this case we denote the orthogonal projection by $P_C(x)$.
We point out that in this case $\dist_f(x,C)^{2}$ and the usual $\dist(x,C)^{2}$ differ by a factor of $2$, but we prefer this slight inconsistency to incorporating the factor into the definition of $\dist_f$.
The Bregman projection can also be characterized by a variational inequality.

\begin{lem}[{\cite[Lemma 2.2]{LSW14}}] \label{lem:BP}
Let $f:\RR^n \to \RR$ be strongly convex.
Then a point $\hat{x} \in C$ is the Bregman projection of $x$ onto $C$ with respect to $f$ and $x^* \in \partial f(x)$ iff there is some $\hat{x}^* \in \partial f(\hat{x})$ such that one of the following equivalent conditions is fulfilled
\ben
\scp{\hat{x}^*-x^*}{y-\hat{x}} \ge 0 \quad \mbox{for all} \quad y \in C
\een
\ben
D_f^{\hat{x}^*}(\hat{x},y) \le  D_f^{x^*}(x,y) - D_f^{x^*}(x,\hat{x}) \quad \mbox{for all} \quad y \in C \,.
\een
We call any such $\hat{x}^*$ an \emph{admissible subgradient} for $\hat{x}=\Pi_C^{x^*}(x)$.
\end{lem}

Bregman projections onto affine subspaces and half-spaces can be computed efficiently.

\begin{definition}
Let $A \in \RR^{m \times n}$, $b \in \RR^m$, $u \in \RR^n$ and $\beta \in \RR$.
By $L(A,b)$ we denote the \emph{affine subspace}
\ben
L(A,b):=\set{ x \in \RR^n}{ Ax = b } \,,
\een
by $H(u,\beta)$ the \emph{hyperplane}
\ben
H(u,\beta):=\set{ x \in \RR^n}{ \scp{ u}{ x }= \beta } \,,
\een
and by $H_{\le}(u,\beta)$ the \emph{half-space}
\ben
H_{\le}(u,\beta) :=\set{ x \in \RR^n }{ \scp{u}{x} \le \beta } \,.
\een
\end{definition}

\begin{lem}[{\cite[Lemma 2.4]{LSW14}}] \label{lem:LH}
Let $f:\RR^n \to \RR$ be $\alpha$-strongly convex.
\begin{enumerate}
\item The Bregman projection of $x \in \RR^n$ onto $L(A,b)\not=\emptyset$ is
\ben
\hat{x}:=\Pi_{L(A,b)}^{x^*}(x)=\nabla f^*(x^*-A^T \hat{w}) \,,
\een
where $\hat{w}\in \RR^m$ is a solution of
\ben
\min_{w \in \RR^m} f^*(x^*-A^T w) + \scp{ w }{ b } \,.
\een
Moreover, an admissible subgradient for $\hat{x}$ is $\hat{x}^*:=x^*-A^T \hat{w}$.
If $A$ has full row rank then for all $y\in L(A,b)$ we have
\ben
D_f^{\hat{x}^*}(\hat{x},y) \le D_f^{x^*}(x,y) - \frac{\alpha}{2} \cdot \norm[2]{(AA^T)^{-\frac{1}{2}}(Ax-b)}^2 \,.
\een
\item The Bregman projection of $x \in \RR^n$ onto $H(u,\beta)$ with $u \not=0$ is
\ben
\hat{x} := \Pi_{H(u,\beta)}^{x^*}(x)=\nabla f^*(x^*- \hat{t} \cdot u) \,,
\een
where $\hat{t} \in \RR$ is a solution of
\ben
\min_{t \in \RR} f^*(x^*-t  \cdot u) + t \cdot \beta \,.
\een
Moreover, an admissible subgradient for $\hat{x}$ is $\hat{x}^*:=x^* - \hat{t} \cdot u$ and for all $y\in H(u,\beta)$ we have
\ben
D_f^{\hat{x}^*}(\hat{x},y) \le D_f^{x^*}(x,y) - \frac{\alpha}{2} \cdot \frac{(\scp{u}{x} - \beta)^2} {\norm[2]{u}^2} \,. \label{eq:D-decreasing-H}
\een
If $x \notin H_{\le}(u,\beta)$ then we necessarily have $\hat{t}>0$, $\Pi_{H_{\le}(u,\beta)}^{x^*}(x)=\hat{x}$ and the above inequality holds for all $y\in H_{\le}(u,\beta)$.
\end{enumerate}
\end{lem}

\section{Bounded linear regularity and error bounds}

As in~\cite{BB96} for the case of metric projections, we will establish convergence rates with Bregman projections under the assumption of bounded linear regularity. 
By $\rint(C)$ we denote the relative interior of a subset $C \subset \RR^n$.

\begin{definition} \label{def:LinReg}
Let $C_1,\ldots C_r \subset \RR^n$ be closed convex sets with nonempty intersection $C:=\bigcap_{i=1}^r C_i$.
\begin{enumerate}
\item The collection $\{C_1,\ldots C_r\}$ is called \emph{boundedly linearly regular}, if for every $R>0$ there exists $\gamma>0$ such that for all $x \in B_R$ we have
\ben
\dist(x,C)^2 \le \gamma \cdot \sum_{i=1}^r \dist(x,C_i)^2 \,,
\een
and it is called \emph{linearly regular}, if such an estimate holds globally for all $x \in \RR^n$.
\item The collection $\{C_1,\ldots C_r\}$ satisfies the \emph{standard constraint qualification}, if there exists $q \in \{0,\ldots,r\}$ such that $C_{q+1},\ldots,C_r$ are polyhedral and
\ben
\bigcap_{i=1}^q \rint(C_i) \cap \bigcap_{i=q+1}^r C_i \not= \emptyset\,.
\een
\end{enumerate}
\end{definition}

\begin{thm}[Corollary 3 and 6 in~\cite{BBL99}] \label{thm:LinReg}
If the collection $\{C_1,\ldots C_r\}$ satisfies the standard constraint qualification then it is boundedly linearly regular.
And if $C$ is also bounded, then $\{C_1,\ldots C_r\}$ is linearly regular.
\end{thm}

By Lemma~\ref{lem:D}, and since $\dist_f^{x^*}(x,C)^2 \le D_f^{x^*}\big(x,P_{C}(x)\big)$, we can immediately bound the Bregman distance by the metric distance.

\begin{lem} \label{lem:BregmanBound}
Let $f:\RR^n \to \RR$ be strongly convex.
\begin{enumerate}
\item For all $x \in \RR^n$, $x^* \in \partial f(x)$ and $y^* \in \partial f\big(P_{C}(x)\big)$ we have
\ben
\dist_f^{x^*}(x,C)^2 \le \norm[2]{x^*-y^*} \cdot \dist(x,C) \,.
\een
\item If $f$ has a $L$-Lipschitz-continuous gradient then we have for all $x \in \RR^n$
\ben
\dist_f(x,C)^2 \le \tfrac{L}{2} \cdot \dist(x,C)^2 \,.
\een
\end{enumerate}
\end{lem}

In general, it is not obvious how to extend the second (and better) estimate to non-differentiable funtions $f$, because we lack an inequality like $\norm[2]{x^*-y^*}  \le L \cdot \norm[2]{x-y}$.
However, we can achieve the better estimate for convex piecewise linear-quadratic $f$.
The result is based on the following lemma, which exploits the fact that the subgradients on the sets $F_x$ are closely related, cf. Definition~\ref{def:cpq}.

\begin{lem} \label{lem:Dcpq} 
Let $f:\RR^n \to \RR$ be strongly convex piecewise linear-quadratic and $C \subset \RR^n$ be closed convex.
Then for all $R>0$ there exists $L>0$ such that for all $x \in B_R$ and $x^* \in \partial f(x)$ we have
\[
\dist_f^{x^*}(x,C)^2  \le \begin{cases} L \cdot \dist(x,C)^2 &, F_x \cap C=\emptyset \\ L \cdot \dist(x,F_x \cap C)^2 &, F_x \cap C \not=\emptyset \end{cases} \,.
\]
\end{lem}

\begin{proof}
Since $B_R$ is compact we have $\dist(B_R \cap F_x,C) > 0$ for all $x \in B_R$ with $F_x \cap C=\emptyset$.
Since there are only finitely many different sets $F_x$ it follows that
\[
d:=\min \set{\dist(B_R \cap F_x,C)}{\mbox{$x \in B_R$ with $F_x \cap C=\emptyset$}} > 0 \,.
\]
Furthermore there is a constant $c>0$ such that $\norm[2]{x^*-y^*} \le c$ for all $x \in B_R$, $x^* \in \partial f(x)$ and $y^* \in \partial f\big(P_C(x)\big)$.
Let $x \in B_R$ and $x^* \in \partial f(x)$.
By Theorem~\ref{thm:partialfcpq} there are $\lambda_i \in [0,1]$ with $\sum_{i \in I_f(x)} \lambda_i=1$ such that
\[
x^*=\sum_{i \in I_f(x)} \lambda_i \cdot (A_i x +a_i) \,.
\]
In case $F_x \cap C =\emptyset$ we have $\dist(x,C)\ge d$, and hence by Lemma~\ref{lem:BregmanBound} we get
\ben
\dist_f^{x^*}(x,C)^2 \le  \norm[2]{x^*-y^*} \cdot \dist(x,C) \le \frac{c}{d} \cdot \dist(x,C)^2 \,.
\een
In case $F_x \cap C \not=\emptyset$ we set $\hat{x}:=P_{F_x \cap C}(x)$.
Since $\hat{x} \in F_x$ we have $I_f(x) \subset I_f(\hat{x})$, and therefore we can choose the following subgradient of $f$ at $\hat{x}$,
\ben
\hat{x}^* :=\sum_{i \in I_f(x)} \lambda_i \cdot (A_i \hat{x} +a_i)
\een
with the same $\lambda_i$ as for $x^*$.
We set $L_f:=\max\set{\norm[2]{A_i}}{i \in I}$ and estimate
\ben
\scp{x^*-\hat{x}^*}{x-\hat{x}} = \sum_{i \in I_f(x)} \lambda_i \cdot \scp{A_i(x-\hat{x})}{x-\hat{x}} \le  L_f\cdot \norm[2]{x-\hat{x}}^2\,,
\een
which yields $\dist_f^{x^*}(x,C)^2 \le \scp{x^*-\hat{x}^*}{x-\hat{x}}  \le  L_f\cdot \dist(x,F_x \cap C)^2$.
\end{proof}

Now we can prove the main theorem of this section.

\begin{thm} \label{thm:BregmanBoundcpq}
Let $f:\RR^n \to \RR$ be strongly convex piecewise linear-quadratic, and let $C \subset \RR^n$ be closed convex such that the collections $\{F_x,C\}$ are boundedly linearly regular for all $x \in \RR^n$ with $F_x \cap C \not=\emptyset$.
Then for all $R>0$ there exists $L>0$ such that for all $x \in B_R$ and $x^* \in \partial f(x)$ we have
\ben
\dist_f^{x^*}(x,C)^2  \le L \cdot \dist(x, C)^2\,.
\een
\end{thm}

\begin{proof}
The assertion immediately follows from Lemma~\ref{lem:Dcpq} and Definition~\ref{def:LinReg}, because $\dist(x, F_x)=0$.
\end{proof}

\begin{rem}
If $C$ is polyhedral then by Theorem~\ref{thm:LinReg} all collections $\{F_x,C\}$ are boundedly linearly regular. 
\end{rem}

For the split feasibility problem we also need the following generalization of Hoffmann's error bound~\cite{Hof52} to possibly non-polyhedral sets, which are defined by convex constraints in the range $\Rcal(A)$ of a matrix $A$.

\begin{lem} \label{lem:errorbound}
Let the convex set $C \subset \RR^n$ have the form $C=\set{x \in \RR^n}{Ax \in Q}$ with $A \in \RR^{m \times n}$ and $Q \subset \RR^m$ closed convex such that the collection $\{Q,\Rcal(A)\}$ is boundedly linearly regular.
Then for every $R>0$ there exists $\gamma>0$ such that for all $x \in B_R$ we have
\ben
\dist(x,C) \le \gamma \cdot \dist(Ax, Q)\,.
\een
\end{lem}

\begin{proof}
In case $A=0$ (and $0 \in Q$) we have $C=\RR^n$ and hence the assertion holds trivially.
Otherwise let $\sigma_{min}>0$ be the smallest positive singular value of $A$, and let $R>0$.
Since $\{Q,\Rcal(A)\}$ is boundedly linearly regular, there exists $\gamma>0$ such that for all $x \in B_R$ we have
\[
\dist\big(Ax,Q \cap \Rcal(A)\big) \le \gamma \cdot \dist(Ax, Q) \,.
\]
To $x \in B_R$ we find some $\hat{x} \in C$ such that $A\hat{x}=P_{Q \cap \Rcal(A)}(Ax)$.
Since $\hat{x} + \Ncal(A) \subset C$ for the nullspace $\Ncal(A)$ of $A$ we get
\begin{align*}
\dist(x,C) & \le \norm[2]{x-P_{\hat{x} + \Ncal(A)}(x)} = \norm[2]{(x-\hat{x})-P_{\Ncal(A)}(x-\hat{x})} \\
& \le \tfrac{1}{\sigma_{min}} \cdot \norm[2]{Ax-A\hat{x}} = \tfrac{1}{\sigma_{min}} \cdot \dist\big(Ax,Q \cap \Rcal(A)\big) \\
& \le \tfrac{\gamma}{\sigma_{min}} \cdot \dist(Ax, Q) \,,
\end{align*}
from which the assertion follows.
\end{proof}

Note that for polyhedral sets $Q$ the collection $\{Q,\Rcal(A)\}$ is always boundedly linearly regular.
Moreover in this case the classical result of Hoffmann holds globally for all $x \in \RR^n$, cf.~\cite{Hof52}.
For non-polyhedral sets $Q$ the assertion holds if $\rint(Q) \cap \Rcal(A) \not= \emptyset$, cf.~Theorem~\ref{thm:LinReg}.
Indeed, if this condition is not fulfilled, the assertion cannot be guaranteed in general, as the following counterexample demonstrates: For $Q=\set{x \in \RR^2}{\norm[2]{x-(0,1)^T} \le 1}$ and $A=\begin{pmatrix} 1 & 0\\0 & 0 \end{pmatrix}$ we have $Q \cap \Rcal(A)=\{0\}$, $C=\{0\} \times \RR$ and hence for $x_1>0$ we get
\[
\frac{\dist(A(x_1,0)^T, Q)}{\dist((x_1,0)^T,C)} = \frac{\sqrt{1+x_1^2}-1}{x_1} = \frac{x_1}{\sqrt{1+x_1^2}+1} \longrightarrow 0 \quad \mbox{for} \quad x_1 \searrow 0 \,.
\]
Finally we concentrate on feasible linearly constrained optimization problems,
\be
\min_{x \in \RR^n} f(x) \quad \mbox{s.t.} \quad Ax=b \label{eq:L}
\ee
like in~\eqref{eq:BP} or~\eqref{eq:N}.
If the objective function $f$ is strongly convex then~\eqref{eq:L} has a unique solution $\hat{x}$ which fulfills $\partial f(\hat{x}) \cap \Rcal(A^T) \not=\emptyset$, and hence coincides with the Bregman projection $\Pi_{L(A,b)}^{x^*}(x)$ with respect to $f$ for all $x \in \RR^n$ with $x^* \in \partial f(x) \cap \Rcal(A^T) \not=\emptyset$, cf. Lemma~\ref{lem:LH}~(a).
As a consequence for all such $x$, $x^*$ we have $\dist_f^{x^*}\big(x,L(A,b)\big)^2=D_f^{x^*}(x,\hat{x})$.
Our next aim is an error bound of the form $D_f^{x^*}(x,\hat{x}) \le \gamma \cdot \norm[2]{Ax-b}^2$.
For piecewise linear-quadratic or differentiable $f$ this immediately follows from Lemma~\ref{thm:BregmanBoundcpq} and~\ref{lem:BregmanBound}~(b) and Hoffmann's error bound.
But we will also achieve this result under weaker assumtions.
To clarify these assumtions we need the concept of calmness of a set-valued mapping~\cite{RW09}.

\begin{definition}
A set-valued mapping $S:\RR^n \rightrightarrows \RR^m$ is \emph{calm} at $\hat{x} \in \RR^n$ if $S(\hat{x}) \not= \emptyset$ and there are constants $\epsilon,L>0$ such that
\ben
S(x) \subset S(\hat{x}) + L \cdot \norm[2]{x-\hat{x}} \cdot B_1 \quad,\quad \norm[2]{x-\hat{x}}\le \epsilon\,.
\een
\end{definition}

\begin{exmp} \label{exmp:calm}
\begin{enumerate}
\item Any \emph{polyhedral multifunction}, i.e. a set-valued mapping whose graph is the union of finitely many polyhedral sets, is calm at each $\hat{x} \in \RR^n$.
In particular this holds for the subdifferential mapping $\partial f(x)$ of a convex piecewise linear-quadratic function $f:\RR^n \to \RR$, see Proposition 1 in~\cite{Rob81}.
\item  Let $\sigma(X) \in \RR^m$ denote the vector of singular values of $X \in \RR^{n_1 \times n_2}$ (with $m=\min\{n_1,n_2\}$), and let $h:\RR^m \to \RR$ be a convex piecewise linear-quadratic function which is \emph{absolutely symmetric}, i.e. $h(x_1,\ldots,x_m)=h\big(|x_{\pi(1)}|, \ldots, |x_{\pi(m)}|\big)$ for any permutation $\pi$ of the indices.
Then the subdifferential mapping of $f(X):=h\big(\sigma(X)\big)$ is calm at each $\hat{X} \in \RR^{n_1 \times n_2}$.
In particular this holds for the \emph{nuclear norm} $\norm[*]{X}:=\norm[1]{\sigma(X)}$, the \emph{spectral norm} $\norm[2]{X}:=\norm[\infty]{\sigma(X)}$ and $f(X)=\lambda \cdot \norm[*]{X} + \tfrac{1}{2} \cdot \norm[F]{X}^2$.
Furthermore the subdifferential mapping of
\ben
f(X_1,X_2)=\tfrac{1}{2}\cdot \norm[F]{X_1}^2 + \lambda_1 \cdot \norm[*]{X_1} + \tfrac{1}{2}\cdot \norm[F]{X_2}^2 + \lambda_2 \cdot \norm[1]{X_2}
\een
is calm at each $(\hat{X_1},\hat{X_2}) \in \RR^{n_1 \times n_2} \times \RR^{n_1 \times n_2}$, where $\norm[1]{X}$ denotes the $1$-norm of all entries of a matrix $X$, see Example 2.10 in~\cite{Sch16}.
\end{enumerate}
\end{exmp}

Now we can reformulate Theorem 2.12 in~\cite{Sch16} to fit the present context.

\begin{thm} \label{thm:L}
Consider the linearly constrained optimization problem~\eqref{eq:L} with $A \in \RR^{m \times n}$, $b \in \Rcal(A)$, and strongly convex $f:\RR^n \to \RR$.
Let $x_0 \in \RR^n$ and $x_0^*\in \partial f(x_0) \cap \Rcal(A^T)$ be given.
If the subdifferential mapping of $f$ is calm at the unique solution $\hat{x}$ of~\eqref{eq:L} and if the collection $\{\partial f(\hat{x}), \Rcal(A^T)\}$ is linearly regular, then there exists $\gamma>0$ such that for all $x \in \RR^n$ and $x^* \in \partial f(x) \cap \Rcal(A^T)$ with $D_f^{x^*}(x,\hat{x}) \le D_f^{x_0^*}(x_0,\hat{x})$ we have
\ben
\dist_f^{x^*}(x,L(A,b))^2 = D_f^{x^*}(x,\hat{x})\le \gamma \cdot \norm[2]{Ax-b}^2\,.
\een
\end{thm}

\begin{proof}
To obtain the error bound we apply the results of~\cite{Sch16} to the objective function $g(y)=f^*(A^T y)- \scp{b}{y}$ of the unconstrained dual
\ben
\min_{y \in \RR^m} f^*(A^T y)- \scp{b}{y} \,,
\een
which relates to the Bregman distance in the following way by setting $x^*=A^T y$, $x=\nabla f^*(x^*)$ and observing that $\scp{b}{y}=\scp{x^*}{\hat{x}}$,
\ben
D_f^{x^*}(x,\hat{x})=g(y)-g_{min} \,.
\een
It follows from Theorem 2.12 in~\cite{Sch16} that the function $g$ is \emph{restricted strongly convex} on all of its level sets.
Hence, by Lemma 2.2 in~\cite{Sch16}, there exists $\gamma>0$ such that for all $x \in \RR^n$ and $x^* \in \partial f(x) \cap \Rcal(A^T)$ with $D_f^{x^*}(x,\hat{x}) \le D_f^{x_0^*}(x_0,\hat{x})$ we have
\ben
D_f^{x^*}(x,\hat{x})=g(y)-g_{min} \le \gamma \cdot \norm[2]{\nabla g(y)}^2 = \gamma \cdot \norm[2]{Ax-b}^2\,.
\een
\end{proof}

\section{Randomized Bregman Projections for SFP}

The \emph{convex feasibility problem} (CFP) is to find a common point of finitely many closed convex sets $C_i \subset \RR^n$, $i \in I:=\{1,\ldots,m\}$, with nonempty intersection,
\be
\mbox{find} \quad x \in C:=\bigcap_{i\in I} C_i \,. \label{eq:CFP}
\ee
A simple and widely known idea to solve~\eqref{eq:CFP} is to project successively onto the individual sets $C_i$ and we refer to~\cite{BB96} for an excellent introduction.
By now there is a vast literature on CFPs and projection algorithms for their solution, see e.g.~\cite{BB97, BBC03, Bre67, Byr04, CEKB05, ZY05}.
These projection algorithms are most efficient if the projections onto the individual sets are relatively cheap.
Here we concentrate on a special instance of the CFP, also called \emph{split feasibility problem} (SFP)~\cite{CE94,BC01,Byr02, SSL08b}, where some or all of the sets $C_i$ arise by imposing convex constraints $Q_i \subset \RR^{m_i}$ in the range of a matrix $A_i \in \RR^{m_i \times n}$,
\be
C_i= \{ x \in \RR^n \,|\, A_i x \in Q_i \} \,. \label{eq:Q_i}
\ee
In general projections onto such sets can be prohibitively expensive and it is often preferable to use projections onto suitable enclosing halfspaces.
The following lemma shows a construction of such an enclosing halfspace, see~\cite{LSW14}. 

\begin{lem} \label{lem:enclosing-halfspace}
Let $Q \subset \RR^m$ be a nonempty closed convex set and $A \in \RR^{m \times n}$.
Assume that $\tilde{x} \notin C=\set{ x \in \RR^n}{ A x \in Q }$ and set
\ben
w:=A\tilde{x}-P_Q(A\tilde{x}) \quad \mbox{and} \quad \beta:=\scp{A^T w}{ \tilde{x}} - \norm[2]{w}^2 \,.
\een
Then it holds that $A^T w \not=0$, $\tilde{x} \notin H_{\le}(A^T w, \beta)$ and $C \subset H_{\le}(A^T w, \beta)$.
In other words, the hyperplane $H(A^T w,\beta)$ separates $\tilde{x}$ from $C$.
\end{lem}

To solve a split feasibility problem one can proceed as follows:
Let $I_Q \subset I$ be the subset of all indices $i$ belonging to sets of the form~\eqref{eq:Q_i}, and denote by $I_C:=I \setminus I_Q$ the set of the remaining indices.
Encounter the different constraints $C_i$ successively and project the current iterate onto $C_i$ in case $i \in I_C$, or onto an enclosing halfspace according to Lemma~\ref{lem:enclosing-halfspace} and Lemma~\ref{lem:LH}~(b) in case $i \in I_Q$, see~Algorithm~\ref{alg:RBPSFP}.
In~\cite{LSW14} convergence of the iterates to a solution of~\eqref{eq:CFP} was shown for Bregman projections with respect to nondifferentiable functions, and for quite general control sequences $i:\NN \to I$.
The only requirement was that $\big(i(k)\big)_{k \in \NN}$ encounters each index in $I$ infinitely often.\footnote{Because very general control sequences $i:\NN \to I$ besides simple cyclic control fulfill this requirement, the corresponding method was also called \emph{method of random Bregman projections} in~\cite{BB97}.
But such control sequences are not necessarily stochastic objects, in contrast to the situation in the present work.
Hence we use the word \emph{randomized} in Algorithm~\ref{alg:RBPSFP} instead of \emph{random} to distinguish between the cases.}
However, no assertion was made about convergence rates.
Here we follow~\cite{NT14,BN14,AWL15,SV09,CP12,ZF13,LL10,MY13,RT14} and show that a randomized version of the algorithm converges in expectation to a solution of~\eqref{eq:CFP} with an expected (sub-)linear convergence rate.

\begin{algorithm}[h]
  \caption{Randomized Bregman projections for split feasibility problems (RBPSFP)}
  \label{alg:RBPSFP}
  \begin{algorithmic}[1]
    \REQUIRE{starting points $x_0\in\RR^n$, $x_0^* \in \partial f(x_0)$ and probabilities $p_i>0$, $i \in I$}
    \ENSURE{a solution of~\eqref{eq:CFP}}
    \STATE initialize $k =  0$
    \REPEAT
    \STATE choose an index $i_k=i \in I$ at random with probability $p_i>0$
    \IF{$i_k\in I_C$}
    \STATE update $x_{k+1} = \Pi_{C_{i_k}}^{x^*_k}(x_k)$ together with an admissible subgradient $x_{k+1}^* \in \partial f(x_{k+1})$, cf. Lemma~\ref{lem:BP}
    \ELSIF{$i_k\in I_Q$} 
    \STATE set $w_k = A_{i_k} x_k - P_{Q_{i_k}}\big(A_{i_k} x_k\big)$ and $\beta_k = \scp{A_{i_k}^T w_k}{x_k} - \norm[2]{w_k}^2$
    \STATE update $x_{k+1} = \Pi_{H_{\le}(A_{i_k}^T w_k, \beta_k)}^{x^*_k}(x_k)$ with $x_{k+1}^* \in \partial f(x_{k+1})$ as in Lemma~\ref{lem:LH}~(b) \label{algline:breg-proj-halfspace}
    \ENDIF
    \STATE increment $k =  k+1$
    \UNTIL{a stopping criterion is satisfied}
  \end{algorithmic}
\end{algorithm}

\begin{thm} \label{thm:RBPSFP}
Let $f:\RR^n \to \RR$ be $\alpha$-strongly convex.
Consider the SFP~\eqref{eq:CFP} under the assumption that the collections $\{C_1,\ldots,C_r\}$ and $\{Q_i,\Rcal(A_i)\}$ for each $i \in I_Q$ are boundedly linearly regular.
Then for any starting points $x_0\in\RR^n$ and $x_0^* \in \partial f(x_0)$ the iterates $x_k$ and $x_k^*$ of Algorithm~\ref{alg:RBPSFP} remain bounded, the Bregman distances to $C$ decrease monotonically,
\[
\dist_f^{x_{k+1}^*}(x_{k+1} ,C) \le \dist_f^{x_k^*}(x_k ,C) \,,
\]
and converge in expectation to zero, where the expectation is taken with respect to the probability distribution $p_i>0$, $i \in I$.
The expected rate of convergence is at least sublinear: There is a constant $c >0$ such that
\[
\EE \left[\dist(x_k,C)\right] \le \frac{c}{\sqrt{k}} \,.
\]
\end{thm}

\begin{proof}
At first we consider the case $i_k \in I_C$.
By Lemma~\ref{lem:D} we have
\[
D_f^{x_k^*}(x_k,x_{k+1})\ge \frac{\alpha}{2} \cdot \norm[2]{x_k-x_{k+1}}^2 \ge \frac{\alpha}{2} \cdot \dist(x_k,C_{i_k})^2 \,,
\]
and together with Lemma~\ref{lem:BP} we can estimate for all $x \in C$
\be
D_f^{x_{k+1}^*}(x_{k+1} ,x)  \le D_f^{x_k^*}(x_k ,x) - \frac{\alpha}{2} \cdot \dist(x_k,C_{i_k})^2\,. \label{eq:decreaseIC}
\ee
Now we consider the case $i_k \in I_Q$.
By Lemma~\ref{lem:enclosing-halfspace} we have $C \subset H_{\le}(A_{i_k}^T w_k, \beta_k)$, and together with Lemma~\ref{lem:LH}~(b) we can estimate for all $x \in C$
\be
D_f^{x_{k+1}^*}(x_{k+1} ,x) \le D_f^{x_k^*}(x_k ,x) - \frac{\alpha}{2 \cdot \norm[2]{A_{i_k}}^2} \cdot \norm[2]{A_{i_k} x_k - P_{Q_{i_k}}\big(A_{i_k} x_k\big)}^2 \,. \label{eq:decreaseIQ}
\ee
We fix some $x \in C$ and conclude from~\eqref{eq:decreaseIC},~\eqref{eq:decreaseIQ} and Lemma~\ref{lem:D} that both $x_k$ and $x_k^*$ remain bounded.
Hence by Lemma~\ref{lem:errorbound} and the bounded linear regularity of all $\{Q_i,\Rcal(A_i)\}$, $i \in I_Q$, there exist $\gamma_i>0$ such that for all $k$ we have
\[
\dist(x_k,C_i) \le \gamma_i \cdot \norm[2]{A_{i_k} x_k - P_{Q_{i_k}}\big(A_{i_k} x_k\big)}\,.
\]
Inserting this estimate into~\eqref{eq:decreaseIQ} we get
\[
D_f^{x_{k+1}^*}(x_{k+1} ,x) \le D_f^{x_k^*}(x_k ,x) - \frac{\gamma_i^2 \cdot\alpha}{2 \cdot \norm[2]{A_{i_k}}^2} \cdot \dist(x_k, C_{i_k})^2 \,.
\]
Together with~\eqref{eq:decreaseIC} this implies that the Bregman distances decrease monotonically, and that there is a constant $c>0$ such that
\be
\dist_f^{x_{k+1}^*}(x_{k+1} ,C)^2  \le \dist_f^{x_k^*}(x_k ,C)^2 - c \cdot \dist(x_k,C_{i_k})^2\,. \label{eq:decrease}
\ee
For the moment we fix the values of the indices $i_0,\ldots,i_{k-1}$ and consider only $i_k$ as a random variable with values in $I$.
Taking the expectation on both sides of~\eqref{eq:decrease} conditional to the values of the indices $i_0,\ldots,i_{k-1}$ yields
\[
\EE \left[\dist_f^{x_{k+1}^*}(x_{k+1} ,C)^2 \,\middle|\, i_0,\ldots,i_{k-1} \right] \le \dist_f^{x_k^*}(x_k ,C)^2 - \sum_{i \in I} p_i \cdot c \cdot \dist(x_k,C_i)^2\,.
\]
By boundedness of $x_k$ and bounded linear regularity of the collection $\{C_1,\ldots,C_m\}$ there is $\gamma>0$ such that for all $k$ we have
\begin{equation}\label{eq:decrease2}
  \EE \left[\dist_f^{x_{k+1}^*}(x_{k+1} ,C)^2 \,\middle|\, i_0,\ldots,i_{k-1} \right] \le \dist_f^{x_k^*}(x_k ,C)^2 - \gamma \cdot \dist(x_k,C)^2\,.
\end{equation}
Furthermore, by Lemma~\ref{lem:BregmanBound}~(a) there is $L>0$ such that for all $k$ we have $\dist_f^{x^*}(x_k,C)^4 \le L \cdot \dist(x_k,C)^2$, and hence we get
\[
\EE \left[\dist_f^{x_{k+1}^*}(x_{k+1} ,C)^2 \,\middle|\, i_0,\ldots,i_{k-1} \right] \le \dist_f^{x_k^*}(x_k ,C)^2 - \tfrac{\gamma}{L} \cdot \dist_f^{x_k^*}(x_k ,C)^4 \,.
\]
Now we consider all indices $i_0,\ldots,i_k$ as random variables with values in $I$, and take the full expectation on both sides,
\begin{align*}
\EE \left[\dist_f^{x_{k+1}^*}(x_{k+1} ,C)^2 \right]  &\le  \EE \left[\dist_f^{x_k^*}(x_k ,C)^2\right] - \tfrac{\gamma}{L} \cdot  \EE \left[\dist_f^{x_k^*}(x_k ,C)^4\right] \\ 
&\le \EE \left[\dist_f^{x_k^*}(x_k ,C)^2\right] - \tfrac{\gamma}{L} \cdot \left(\EE \left[\dist_f^{x_k^*}(x_k ,C)^2\right]\right)^2 \,. \nonumber
\end{align*}
We set $d_k:=\EE \left[\dist_f^{x_k^*}(x_k ,C)^2\right]$.
Then we have $d_{k+1}\leq d_{k}  - \frac{\gamma}{L}d_{k}^{2}$.
We observe that $d_{k}$ is decreasing and by rearranging the inequality to
\[
\frac{1}{d_{k+1}}\geq \frac1{d_{k}} + \frac{\gamma}{L}\frac{d_{k}}{d_{k+1}}\geq \frac1{d_{k}} + \frac{\gamma}{L}
\]
we obtain $\frac{1}{d_{k+1}}\geq \frac1{d_{0}} + \frac{\gamma}{L}(k+1)$, and we conclude $d_{k}\leq \frac{Ld_{0}}{L + \gamma d_{0} \cdot k}$ as desired.
The expected sublinear convergence rates for $\dist(x_k,C)$ now follow from the estimate $\EE \left[\dist(x_k,C)\right] \le \sqrt{\tfrac{2}{\alpha}} \cdot \EE \left[\dist_f^{x_k^*}(x_k ,C)\right]$, cf.~Lemma~\ref{lem:D}.
\end{proof}

\begin{rem}\label{rem:inexact-proj}
According to Lemma~\ref{lem:LH}~(b) the computation of the Bregman projection $x_{k+1} = \Pi_{H_{\le}(A_{i_k}^T w_k, \beta_k)}^{x^*_k}(x_k)$ onto the halfspace $H_{\le}(A_{i_k}^T w_k, \beta_k)$ in step~\ref{algline:breg-proj-halfspace} of Algorithm~\ref{alg:RBPSFP} amounts to an exact linesearch.
In practice, this is feasible only in special cases, e.g. for $f(x)=\|x\|_2^2$ or $f(x)=\lambda \cdot \|x\|_1 + \frac{1}{2} \|x\|_2^2$.
But the assertions of Theorem~\ref{thm:RBPSFP} and the next two theorems remain true for inexact linesearches as well, cf.~\cite{LSW14}.
In particular, we may choose
\[
t_k := \alpha \cdot \tfrac{\norm[2]{w_k}^2}{\norm[2]{A_{i_k}^T w_k}^2} \quad,\quad x_{k+1}^*:=x_k^*-t_k \cdot A_{i_k}^T w_k \quad,\quad x_{k+1}=\nabla f^*(x_{k+1}^*) \,.
\]
\end{rem}

For piecewise linear-quadratic or differentiable $f$ the expected rate of convergence is even linear.

\begin{thm} \label{thm:RBPSFPlinear}
If $f$ is piecewise linear-quadratic or has a Lipschitz-continuous gradient, then under the assumptions of Theorem~\ref{thm:RBPSFP} the expected rate of convergence is linear: There are constants $q \in (0,1)$ and $c>0$ such that
\[
\EE \left[\dist_f^{x_{k+1}^*}(x_{k+1} ,C)^2\right] \le q \cdot \EE \left[\dist_f^{x_k^*}(x_k ,C)^2\right] \,,
\]
and hence
\[
\EE \left[\dist(x_k,C)\right] \le  c \cdot q^\frac{k}{2}\,.
\]
\end{thm}

\begin{proof}
By Theorem~\ref{thm:BregmanBoundcpq} and Lemma~\ref{lem:BregmanBound}~(b) respectively, there is $L>0$ such that for all $k$ we have $\dist_f^{x_k^*}(x_k ,C)^2 \le L\cdot \dist(x_k,C)^2$.
Hence, using this in~\eqref{eq:decrease2} in the proof of  Theorem~\ref{thm:RBPSFP} we get
\[
\EE \left[\dist_f^{x_{k+1}^*}(x_{k+1} ,C)^2 \right] \le  \left(1-\tfrac{\gamma}{L}\right) \cdot \EE \left[\dist_f^{x_k^*}(x_k ,C)^2\right] \,,
\]
from which the linear convergence rates follow.
\end{proof}

Finally we turn to linearly constrained optimization problems.

\begin{thm} \label{thm:RBPSFP-L}
Consider the linearly constrained optimization problem~\eqref{eq:L} under the assumptions of Theorem~\ref{thm:L}.
Let $I_1,\dots, I_r$ be a covering of $\{1,\dots,m\}$ (not necessarily disjoint), denote by $A_i$ the matrix consisting of the rows of $A$ indexed by $I_i$, and let $b_i$ denote the vector consisting of the entries of $b$ indexed by $I_i$.
The constraints $A_i x=b_i$ may be considered both as constraints with $i \in I_C$, cf. Lemma~\ref{lem:LH}~(a), or with $i \in I_Q$ and $Q_i=\{b_i\}$.
If the initial values are chosen as $x_0^* \in \Rcal(A^T)$ and $x_0=\nabla f^*(x_0^*)$ then the iterates of Algorithm~\ref{alg:RBPSFP} converge in expectation to the solution $\hat{x}$ of~\eqref{eq:L}.
The expected rate of convergence is linear: There are constants $q \in (0,1)$ and $c>0$ such that
\[
\EE \left[D_f^{x_{k+1}^*}(x_{k+1} ,\hat{x})\right] \le q \cdot \EE \left[D_f^{x_k^*}(x_k ,\hat{x})\right] \,,
\]
and hence
\[
\EE \left[\norm{x_k-\hat{x}}\right] \le  c \cdot q^\frac{k}{2}\,.
\]
\end{thm}

\begin{proof}
Since $x_0^* \in \Rcal(A^T)$ and the updates are of the form $x_k^* = x_{k-1}^* - A^T v_k$ for some $v_k \in \RR^m$, we inductively get $x_k^* \in \Rcal(A^T)$ for all $k \ge 0$.
Hence the assertion follows from Theorem~\ref{thm:L} as in the proofs of Theorem~\ref{thm:RBPSFP} and~\ref{thm:RBPSFPlinear}.
\end{proof}

\section{Linear convergence of the Randomized Sparse Kaczmarz method}
\label{sec:lin-conv-sparse-kaczmarz}

Here we show how to apply Theorem~\ref{thm:RBPSFP-L} to obtain linear convergence of the Randomized Sparse Kaczmarz method.
As illustrated in~\cite{LWSM14}, the Sparse Kaczmarz
method~\eqref{eq:sparse-kaczmarz-iter} can be considered as a special case of
Algorithm~\ref{alg:RBPSFP} applied to the regularized Basis Pursuit problem~\eqref{eq:BP}.
The objective function
\be
f(x) = \lambda\norm[1]{x} + \tfrac{1}{2}\norm[2]{x}^{2} \label{eq:fBP}
\ee
is $1$-strongly convex and also piecewise linear-quadratic with $\nabla f^{*}(x^{*}) = S_{\lambda}(x^{*})$.
We formulate the constraint $Ax = b$ with sets $Q_{i} = \{b_{i}\}$ and mappings $A_{i} = a_{i}^{T}$ with the rows $a_{i}^{T}$ of $A$, $i\in \{1,\dots,m\}$.
Step 7 in Algorithm~\ref{alg:RBPSFP} then reads as
\[
w_{k}  = \scp{a_{i_{k}}}{x_{k}} - b_{i_{k}} \quad, \quad \beta_{k}  =\scp{a_{i_{k}}w_{k}}{x_{k}}-|w_{k}|^{2} \,.
\]
According to Lemma~\ref{lem:LH}, the Bregman projection $x_{k+1} = \Pi_{H(A_{i_k}^T w_k, \beta_k)}^{x^*_k}(x_k)$ in Step 8 can be computed as
\[
x_{k+1} = \nabla f^{*}(x_{k}^{*}-t_{k} \cdot a_{i_{k}}\cdot w_{k}) = S_{\lambda}\big(x_{k}^{*}- t_{k} \cdot (\scp{a_{i_{k}}}{x_{k}}-b_{i_{k}}) \cdot a_{i_{k}}\big)
\]
with an appropriate stepsize $t_{k}$.
Now we use the inexact stepsize according to Remark~\ref{rem:inexact-proj} with $\alpha=1$, namely
\[
t_{k} = \frac{|w_{k}|^{2}}{\norm[2]{a_{i_{k}}w_{k}}^{2}} = \frac{1}{\norm[2]{a_{i_{k}}}^{2}}.
\]
Hence, we do not need the quantity $\beta_{k}$ to perform the iteration, and the full step reads as
\[
  x_{k+1}^{*} = x_{k}^{*} - \tfrac{\scp{a_{i_{k}}}{x_{k}}-b_{i_{k}}}{\norm[2]{a_{i_{k}}}^{2}} \cdot a_{i_{k}} \quad, \quad x_{k+1} = S_{\lambda}(x_{k+1}^{*}) \,.
\]
We recover the Randomized Sparse Kaczmarz method, which we state here as Algorithm~\ref{alg:RSK}.

\begin{algorithm}[h]
  \caption{Randomized Sparse Kaczmarz method (RSK)}
  \label{alg:RSK}
  \begin{algorithmic}[1]
    \REQUIRE{starting points $x_0=x_{0}^{*}=0\in\RR^n$, matrix $A\in\RR^{m\times n}$, vector $b\in\RR^{m}$ such that $Ax=b$ is consistent, and probabilities $p_i>0$, $i\in \{1,\dots,m\}$}
    \ENSURE{the solution of~$\min_{x \in \RR^n}\lambda\norm[1]{x} + \tfrac12\norm[2]{x}^{2}$ s.t. $Ax=b$}
    \STATE initialize $k =  0$
    \REPEAT
    \STATE choose an index $i_k=i \in \{1,\dots,m\}$ at random with probability $p_i>0$
    \STATE set $a_{i_{k}}^T$ to the $i_{k}$-th row of $A$
    \STATE update $x_{k+1}^{*} = x_{k}^{*} - \frac{\scp{a_{i_{k}}}{x_{k}}-b_{i_{k}}}{\norm[2]{a_{i_{k}}}^{2}} \cdot a_{i_{k}}$
    \STATE update $x_{k+1} = S_{\lambda}(x_{k+1}^{*})$
    \STATE increment $k =  k+1$
    \UNTIL{a stopping criterion is satisfied}
  \end{algorithmic}
\end{algorithm}

As already noted in~\cite{LSW14}, it is also possible to perform an exact linesearch for the Sparse Kaczmarz method. To do so, in each step one has to
solve the one-dimensional problem
\begin{equation}
  t_{k}= \argmin_{t \in \RR} f^{*}(x_{k}^{*}-t \cdot a_{i_{k}}) + t \cdot b_{i_{k}}
  \label{eq:sparse-kaczmarz-exact-step}
\end{equation}
which can be done in reasonable time since $f^{*}$ is piecewise linear-quadratic, see~\cite[Section 2.5.2]{LSW14}.
This results in the Exact-Step Randomized Sparse Kaczmarz (ERSK) method, stated as Algorithm~\ref{alg:ERSK}.
Note that ERSK can also be derived by directly considering the constraints as $C_i=H(a_i,b_i)$ and performing exact Bregman projections onto $C_i$.

\begin{algorithm}[h]
  \caption{Exact-Step Randomized Sparse Kaczmarz method (ERSK)}
  \label{alg:ERSK}
  \begin{algorithmic}[1]
    \REQUIRE{starting points $x_0=x_{0}^{*}=0\in\RR^n$, matrix $A\in\RR^{m\times n}$, vector $b\in\RR^{m}$ such that $Ax=b$ is consistent, and probabilities $p_i>0$, $i\in\{1,\dots,m\}$}
    \ENSURE{the solution of~$\min_{x \in \RR^n}\lambda\norm[1]{x} + \tfrac12\norm[2]{x}^{2}$ s.t. $Ax=b$}
    \STATE initialize $k =  0$
    \REPEAT
    \STATE choose an index $i_k=i\in\{1,\dots,m\}$ at random with probability $p_i>0$
    \STATE set $a_{i_{k}}^T$ to the $i$-th row of $A$
    \STATE calculate $t_{k} = \argmin_{t \in \RR} f^{*}(x_{k}^{*}-t \cdot a_{i_{k}}) + t \cdot b_{i_{k}}$
    \STATE update $x_{k+1}^{*} = x_{k}^{*} - t_{k} \cdot a_{i_{k}}$
    \STATE update $x_{k+1} = S_{\lambda}(x_{k+1}^{*})$
    \STATE increment $k =  k+1$
    \UNTIL{a stopping criterion is satisfied}
  \end{algorithmic}
\end{algorithm}

As a consequence of Theorem~\ref{thm:RBPSFP-L} we can conclude the following:
\begin{cor}\label{cor:lin-conv-sparse-kaczmarz}
  Let $A\in \RR^{m\times n}$ and $b\in\RR^{m}$ be in the range of $A$
    and let $\lambda>0$. Then both the RSK method from
    Algorithm~\ref{alg:RSK} and the ERSK method from
    Algorithm~\ref{alg:ERSK} converge in expectation to the unique solution $\hat{x}$ of
    \[
    \min_{x \in \RR^n}\lambda\norm[1]{x} + \tfrac12\norm[2]{x}^{2}\quad s.t.\quad Ax=b
    \]
    at a linear rate, i.e. in both cases there exist $q\in(0,1)$ and $c>0$ such that
    \[
   \EE \left[\norm{x_k-\hat{x}}\right] \le  c \cdot q^\frac{k}{2}\,.
    \]
\end{cor}

Expected linear convergence for a randomized and smoothed Sparse Kaczmarz method was also shown in~\cite{Pet15}.
There the objective function~\eqref{eq:fBP} was replaced by
\be
f_\epsilon(x) = \lambda \cdot r_\epsilon(x) + \tfrac{1}{2}\norm[2]{x}^{2} \label{eq:smoothedfBP}
\ee
with $\epsilon>0$ and $r_\epsilon(x)$ beeing the Moreau envelope of $\norm[1]{x}$,
\[
r_\epsilon(x) = \sum_{i=1}^n \begin{cases} |x_i|-\tfrac{\epsilon}{2} &, |x_i|>\epsilon \\ \tfrac{x_i^2}{2 \epsilon} &, |x_i| \le \epsilon \end{cases} \,.
\]
The function $f_\epsilon$ is $1$-strongly convex and has a Lipschitz-continuous gradient.
Hence linear convergence is also guaranteed by Theorem~\ref{thm:L}.
But as shown above, Theorem~\ref{thm:L} also allows us to prove this result without smoothing the objective function.
Of course this also holds for the Randomized Block Sparse Kaczmarz method considered in~\cite{Pet15} by applying Theorem~\ref{thm:L} with a covering $I_1,\dots, I_r$ of $\{1,\dots,m\}$.

\section{Numerical examples}
\label{sec:numerics}
 
In two experiments we illustrate the impact of the Randomized Sparse Kaczmarz method versus the (non-sparse) Randomized Kaczmarz and the (non-randomized) Sparse Kaczmarz method.

\subsection{Sparse vs. non-sparse Randomized Kaczmarz}

We constructed overdetermined linear systems with Gaussian matrices
$A\in\RR^{m\times n}$ for $m \ge n$, and sparse solutions $\hat{x}\in\RR^{n}$ with
corresponding right hand sides $b=A\hat{x} \in \RR^{m}$ and also respective noisy
right hand sides $b^{\delta}$.  We ran the usual Randomized Kaczmarz method (RK), the Randomized Sparse Kaczmarz method (RSK)
(Algorithm~\ref{alg:RSK}), and the
Exact-Step Randomized Sparse Kaczmarz method (ERSK)
(Algorithm~\ref{alg:ERSK}) on the problem. Note that, since with high probability the matrices $A$ have full rank, in the case of
no noise the solution $\hat{x}$ is unique, and so all methods are expected to converge
to the same solution $\hat{x}$.

Figure~\ref{fig:rk-srk-3} shows the result for a five times
overdetermined and consistent system without noise.
Note that the usual RK performs consistently well over all
trials, while the performance of RSK and ERSK differs drastically between different instances. As denoted by
the quantiles, there are a few instances on which RSK and ERSK are remarkably fast, especially for the
exact-step method, while for other instance they are rather slow.  Also,
the asymptotic linear rate of the medians is fastest for ERSK, and also RSK has a faster asymptotic rate than non-sparse RK.

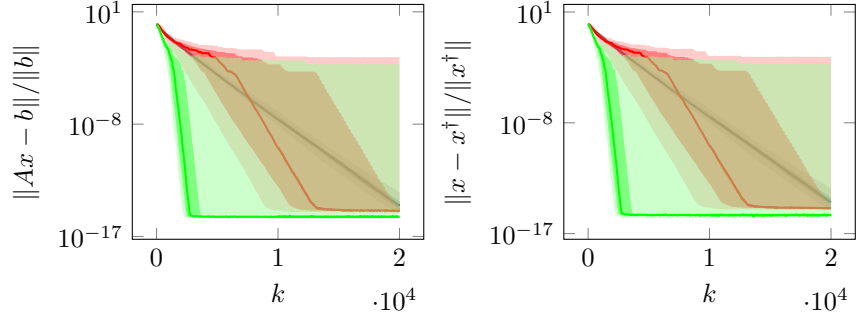
\begin{figure}
  \centering
  \begin{tikzpicture}
    \begin{semilogyaxis}[width=0.45\textwidth,ylabel={$\norm{Ax-b}/\norm{b}$},xlabel={$k$}]
      \addplot[fill=black!40,draw=none,opacity=0.5] table {data/experimentA/maxminres_rk_n200_m1000_s25_rownorms_squared_noiselev0.dat};
      \addplot[fill=black!80,draw=none,opacity=0.5] table {data/experimentA/quant2575res_rk_n200_m1000_s25_rownorms_squared_noiselev0.dat};
      \addplot[thick,draw=black] table {data/experimentA/medianres_rk_n200_m1000_s25_rownorms_squared_noiselev0.dat};
      
      \addplot[fill=red!40,draw=none,opacity=0.5] table {data/experimentA/maxminres_srk_n200_m1000_s25_rownorms_squared_noiselev0.dat};
      \addplot[fill=red!80,draw=none,opacity=0.5] table {data/experimentA/quant2575res_srk_n200_m1000_s25_rownorms_squared_noiselev0.dat};
      \addplot[thick,draw=red] table {data/experimentA/medianres_srk_n200_m1000_s25_rownorms_squared_noiselev0.dat};

      \addplot[fill=green!40,draw=none,opacity=0.5] table {data/experimentA/maxminres_esrk_n200_m1000_s25_rownorms_squared_noiselev0.dat};
      \addplot[fill=green!80,draw=none,opacity=0.5] table {data/experimentA/quant2575res_esrk_n200_m1000_s25_rownorms_squared_noiselev0.dat};
      \addplot[thick,draw=green] table {data/experimentA/medianres_esrk_n200_m1000_s25_rownorms_squared_noiselev0.dat};
    \end{semilogyaxis}
  \end{tikzpicture}
  \begin{tikzpicture}
    \begin{semilogyaxis}[width=0.45\textwidth,ylabel={$\norm{x-\hat{x}}/\norm{\hat{x}}$},xlabel={$k$}]
      \addplot[fill=black!40,draw=none,opacity=0.5] table {data/experimentA/maxminerr_rk_n200_m1000_s25_rownorms_squared_noiselev0.dat};
      \addplot[fill=black!80,draw=none,opacity=0.5] table {data/experimentA/quant2575err_rk_n200_m1000_s25_rownorms_squared_noiselev0.dat};
      \addplot[thick,draw=black] table {data/experimentA/medianerr_rk_n200_m1000_s25_rownorms_squared_noiselev0.dat};
      
      \addplot[fill=red!40,draw=none,opacity=0.5] table {data/experimentA/maxminerr_srk_n200_m1000_s25_rownorms_squared_noiselev0.dat};
      \addplot[fill=red!80,draw=none,opacity=0.5] table {data/experimentA/quant2575err_srk_n200_m1000_s25_rownorms_squared_noiselev0.dat};
      \addplot[thick,draw=red] table {data/experimentA/medianerr_srk_n200_m1000_s25_rownorms_squared_noiselev0.dat};
      
      \addplot[fill=green!40,draw=none,opacity=0.5] table {data/experimentA/maxminerr_esrk_n200_m1000_s25_rownorms_squared_noiselev0.dat};
      \addplot[fill=green!80,draw=none,opacity=0.5] table {data/experimentA/quant2575err_esrk_n200_m1000_s25_rownorms_squared_noiselev0.dat};
      \addplot[thick,draw=green] table {data/experimentA/medianerr_esrk_n200_m1000_s25_rownorms_squared_noiselev0.dat};
    \end{semilogyaxis}
  \end{tikzpicture}
  \caption{Experiment A: Comparison of Randomized Kaczmarz (black)
    Randomized Rparse Kaczmarz (red), and Exact-Step Randomized Sparse
    Kaczmarz (green), $n=200$, $m=1000$, sparsity $s=25$, no
    noise. Left: Plots of relative residual $\|Ax-b\|/\|b\|$, right:
    plots of error $\|x-x^\dag\|/\|x^\dag\|$. Thick line shows median
    over 60 trials, light area is between min and max, darker area
    indicate 25th and 75th quantile.}
  \label{fig:rk-srk-3}
\end{figure}

Figures~\ref{fig:rk-srk-1} and~\ref{fig:rk-srk-2} show the results for
noisy right hand sides.
Figure~\ref{fig:rk-srk-1} uses a two times overdetermined
system with 10\% relative noise, Figure~\ref{fig:rk-srk-2} has the
same noise level and a five times overdetermined system. All methods
consistently stagnate at a residual level which is comparable to the
noise level, however, ERSK achieves this faster than RSK which in turn is faster than RK.
Regarding the reconstruction error, ERSK and RK achieve reconstructions with an error in the size
of the noise level, while SRK achieves an even lower
reconstruction error. The last effect is not explained by our theory.
On an intuitive level one may argue that the Sparse Kaczmarz method
obtains better reconstructions since it incorporates the sparsity of
the solutions, but that the exact steps in the Sparse Kaczmarz method
spoil this advantage by trying to fullfill all equations exactly,
despite the noise. In fact, RSK with inexact
stepsize may be seen as a kind of relaxed Kaczmarz method.

\begin{figure}
  \centering
  \begin{tikzpicture}
    \begin{semilogyaxis}[width=0.45\textwidth,ylabel={$\norm{Ax-b^{\delta}}/\norm{b^{\delta}}$},xlabel={$k$}]
      \addplot[fill=black!40,draw=none,opacity=0.5] table {data/experimentA/maxminres_rk_n200_m400_s25_rownorms_squared_noiselev0.1.dat};
      \addplot[fill=black!80,draw=none,opacity=0.5] table {data/experimentA/quant2575res_rk_n200_m400_s25_rownorms_squared_noiselev0.1.dat};
      \addplot[thick,draw=black] table {data/experimentA/medianres_rk_n200_m400_s25_rownorms_squared_noiselev0.1.dat};
      
      \addplot[fill=red!40,draw=none,opacity=0.5] table {data/experimentA/maxminres_srk_n200_m400_s25_rownorms_squared_noiselev0.1.dat};
      \addplot[fill=red!80,draw=none,opacity=0.5] table {data/experimentA/quant2575res_srk_n200_m400_s25_rownorms_squared_noiselev0.1.dat};
      \addplot[thick,draw=red] table {data/experimentA/medianres_srk_n200_m400_s25_rownorms_squared_noiselev0.1.dat};
      
      \addplot[fill=green!40,draw=none,opacity=0.5] table {data/experimentA/maxminres_esrk_n200_m400_s25_rownorms_squared_noiselev0.1.dat};
      \addplot[fill=green!80,draw=none,opacity=0.5] table {data/experimentA/quant2575res_esrk_n200_m400_s25_rownorms_squared_noiselev0.1.dat};
      \addplot[thick,draw=green] table {data/experimentA/medianres_esrk_n200_m400_s25_rownorms_squared_noiselev0.1.dat};
    \end{semilogyaxis}
  \end{tikzpicture}
  \begin{tikzpicture}
    \begin{semilogyaxis}[width=0.45\textwidth,ylabel={$\norm{x-\hat{x}}/\norm{\hat{x}}$},xlabel={$k$}]
      \addplot[fill=black!40,draw=none,opacity=0.5] table {data/experimentA/maxminerr_rk_n200_m400_s25_rownorms_squared_noiselev0.1.dat};
      \addplot[fill=black!80,draw=none,opacity=0.5] table {data/experimentA/quant2575err_rk_n200_m400_s25_rownorms_squared_noiselev0.1.dat};
      \addplot[thick,draw=black] table {data/experimentA/medianerr_rk_n200_m400_s25_rownorms_squared_noiselev0.1.dat};
      
      \addplot[fill=red!40,draw=none,opacity=0.5] table {data/experimentA/maxminerr_srk_n200_m400_s25_rownorms_squared_noiselev0.1.dat};
      \addplot[fill=red!80,draw=none,opacity=0.5] table {data/experimentA/quant2575err_srk_n200_m400_s25_rownorms_squared_noiselev0.1.dat};
      \addplot[thick,draw=red] table {data/experimentA/medianerr_srk_n200_m400_s25_rownorms_squared_noiselev0.1.dat};
      
      \addplot[fill=green!40,draw=none,opacity=0.5] table {data/experimentA/maxminerr_esrk_n200_m400_s25_rownorms_squared_noiselev0.1.dat};
      \addplot[fill=green!80,draw=none,opacity=0.5] table {data/experimentA/quant2575err_esrk_n200_m400_s25_rownorms_squared_noiselev0.1.dat};
      \addplot[thick,draw=green] table {data/experimentA/medianerr_esrk_n200_m400_s25_rownorms_squared_noiselev0.1.dat};
    \end{semilogyaxis}
  \end{tikzpicture}
  \caption{Experiment A: Comparison of Randomized Kaczmarz (black)
    Randomized Sparse Kaczmarz (red), and Exact-Step Randomized Sparse
    Kaczmarz (green), $n=200$, $m=400$, sparsity $s=25$,
    $10\%$ relative noise. Left: Plots of relative residual
    $\|Ax-b^{\delta}\|/\|b^{\delta}\|$, right: plots of error
    $\|x-x^\dag\|/\|x^\dag\|$. Thick line shows median over 60
    trials, light area is between min and max, darker area indicate
    25th and 75th quantile.}
  \label{fig:rk-srk-1}
\end{figure}
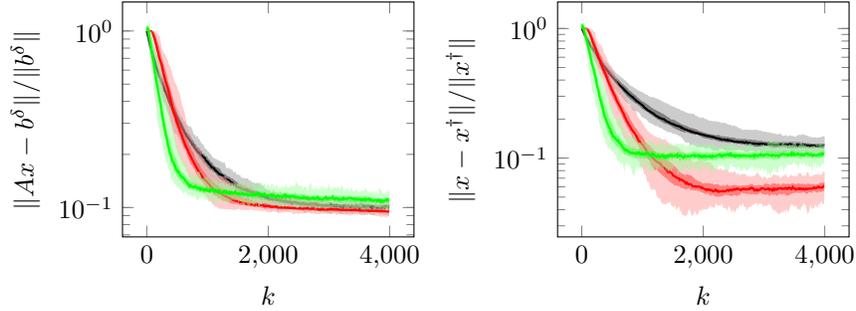

\begin{figure}
  \centering
  \begin{tikzpicture}
    \begin{semilogyaxis}[width=0.45\textwidth,ylabel={$\norm{Ax-b^{\delta}}/\norm{b^{\delta}}$},xlabel={$k$}]
      \addplot[fill=black!40,draw=none,opacity=0.5] table {data/experimentA/maxminres_rk_n200_m1000_s25_rownorms_squared_noiselev0.1.dat};
      \addplot[fill=black!80,draw=none,opacity=0.5] table {data/experimentA/quant2575res_rk_n200_m1000_s25_rownorms_squared_noiselev0.1.dat};
      \addplot[thick,draw=black] table {data/experimentA/medianres_rk_n200_m1000_s25_rownorms_squared_noiselev0.1.dat};
      
      \addplot[fill=red!40,draw=none,opacity=0.5] table {data/experimentA/maxminres_srk_n200_m1000_s25_rownorms_squared_noiselev0.1.dat};
      \addplot[fill=red!80,draw=none,opacity=0.5] table {data/experimentA/quant2575res_srk_n200_m1000_s25_rownorms_squared_noiselev0.1.dat};
      \addplot[thick,draw=red] table {data/experimentA/medianres_srk_n200_m1000_s25_rownorms_squared_noiselev0.1.dat};
      
      \addplot[fill=green!40,draw=none,opacity=0.5] table {data/experimentA/maxminres_esrk_n200_m1000_s25_rownorms_squared_noiselev0.1.dat};
      \addplot[fill=green!80,draw=none,opacity=0.5] table {data/experimentA/quant2575res_esrk_n200_m1000_s25_rownorms_squared_noiselev0.1.dat};
      \addplot[thick,draw=green] table {data/experimentA/medianres_esrk_n200_m1000_s25_rownorms_squared_noiselev0.1.dat};
    \end{semilogyaxis}
  \end{tikzpicture}
  \begin{tikzpicture}
    \begin{semilogyaxis}[width=0.45\textwidth,ylabel={$\norm{x-\hat{x}}/\norm{\hat{x}}$},xlabel={$k$}]
      \addplot[fill=black!40,draw=none,opacity=0.5] table {data/experimentA/maxminerr_rk_n200_m1000_s25_rownorms_squared_noiselev0.1.dat};
      \addplot[fill=black!80,draw=none,opacity=0.5] table {data/experimentA/quant2575err_rk_n200_m1000_s25_rownorms_squared_noiselev0.1.dat};
      \addplot[thick,draw=black] table {data/experimentA/medianerr_rk_n200_m1000_s25_rownorms_squared_noiselev0.1.dat};
      
      \addplot[fill=red!40,draw=none,opacity=0.5] table {data/experimentA/maxminerr_srk_n200_m1000_s25_rownorms_squared_noiselev0.1.dat};
      \addplot[fill=red!80,draw=none,opacity=0.5] table {data/experimentA/quant2575err_srk_n200_m1000_s25_rownorms_squared_noiselev0.1.dat};
      \addplot[thick,draw=red] table {data/experimentA/medianerr_srk_n200_m1000_s25_rownorms_squared_noiselev0.1.dat};
      
      \addplot[fill=green!40,draw=none,opacity=0.5] table {data/experimentA/maxminerr_esrk_n200_m1000_s25_rownorms_squared_noiselev0.1.dat};
      \addplot[fill=green!80,draw=none,opacity=0.5] table {data/experimentA/quant2575err_esrk_n200_m1000_s25_rownorms_squared_noiselev0.1.dat};
      \addplot[thick,draw=green] table {data/experimentA/medianerr_esrk_n200_m1000_s25_rownorms_squared_noiselev0.1.dat};
    \end{semilogyaxis}
  \end{tikzpicture}
  \caption{Experiment A: Comparison of Randomized Kaczmarz (black)
    Randomized Sparse Kaczmarz (red), and Exact-Step Randomized Sparse
    Kaczmarz (green), $n=200$, $m=1000$, sparsity $s=25$,
    $10\%$ relative noise. Left: Plots of relative residual
    $\|Ax-b^{\delta}\|/\|b^{\delta}\|$, right: plots of error
    $\|x-x^\dag\|/\|x^\dag\|$. Thick line shows median over 60
    trials, light area is between min and max, darker area indicate
    25th and 75th quantile.}
  \label{fig:rk-srk-2}
\end{figure}
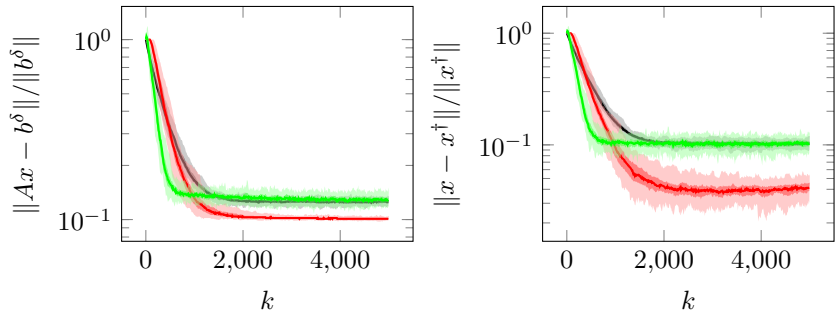

\subsection{Sparse cyclic vs. Randomized Sparse Kaczmarz}
\label{sec:rand-vs-cyclic-sparse}

To investigate the impact of randomization within the Sparse Kaczmarz
framework, we studied an academic tomography problem. We used the
AIRtools toolbox~\cite{HSH12} to create CT-measurement matrices of different
sizes. We used fanbeam geometry throughout and worked with overdetermined
systems, sparse solutions and noisefree right hand sides. We compared RSK with
the cyclic version of the Sparse Kaczmarz method, where we process
the rows of the linear system in their ``natural'' order.
Figure~\ref{fig:sk-srk-1} shows the result for a small problem
with $n=100$ pixels, and Figure~\ref{fig:sk-srk-2} shows the result for
a problem with $n=900$ pixels. In both cases the randomization shows
improvements for the median as well as for the extreme cases.

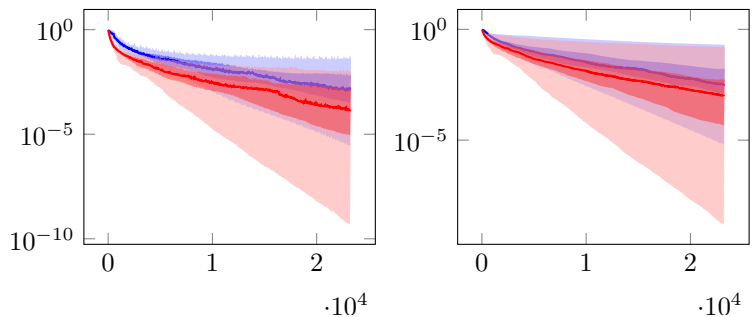
\begin{figure}
  \centering
  \begin{tikzpicture}
    \begin{semilogyaxis}[width=0.45\textwidth]
      \addplot[fill=blue!40,draw=none,opacity=0.5] table {data/experimentB/maxminres_sk_n100_m1164_s20.dat};
      \addplot[fill=blue!80,draw=none,opacity=0.5] table {data/experimentB/quant2575res_sk_n100_m1164_s20.dat};
      \addplot[thick,draw=blue] table {data/experimentB/medianres_sk_n100_m1164_s20.dat};
      
      \addplot[fill=red!40,draw=none,opacity=0.5] table {data/experimentB/maxminres_srk_n100_m1164_s20.dat};
      \addplot[fill=red!80,draw=none,opacity=0.5] table {data/experimentB/quant2575res_srk_n100_m1164_s20.dat};
      \addplot[thick,draw=red] table {data/experimentB/medianres_srk_n100_m1164_s20.dat};
    \end{semilogyaxis}
  \end{tikzpicture}
  \begin{tikzpicture}
    \begin{semilogyaxis}[width=0.45\textwidth]
      \addplot[fill=blue!40,draw=none,opacity=0.5] table {data/experimentB/maxminerr_sk_n100_m1164_s20.dat};
      \addplot[fill=blue!80,draw=none,opacity=0.5] table {data/experimentB/quant2575err_sk_n100_m1164_s20.dat};
      \addplot[thick,draw=blue] table {data/experimentB/medianerr_sk_n100_m1164_s20.dat};
      
      \addplot[fill=red!40,draw=none,opacity=0.5] table {data/experimentB/maxminerr_srk_n100_m1164_s20.dat};
      \addplot[fill=red!80,draw=none,opacity=0.5] table {data/experimentB/quant2575err_srk_n100_m1164_s20.dat};
      \addplot[thick,draw=red] table {data/experimentB/medianerr_srk_n100_m1164_s20.dat};
    \end{semilogyaxis}
  \end{tikzpicture}
  \caption{Experiment B: Sparse Kaczmarz (blue) vs. Sparse
    Randomized Kaczmarz (red), $n=100$, $m=1164$, sparsity $s=20$. Left: Plots of relative residual
    $\|Ax-b\|/\|b\|$, right: plots of error
    $\|x-x^\dag\|/\|x^\dag\|$. Thick line shows median over 40
    trials, light area is between min and max, darker area indicate
    25th and 75th quantile.}
  \label{fig:sk-srk-1}
\end{figure}

\begin{figure}
  \centering
  \begin{tikzpicture}
    \begin{semilogyaxis}[width=0.45\textwidth]
      \addplot[fill=blue!40,draw=none,opacity=0.5] table {data/experimentB/maxminres_sk_n900_m3660_s180.dat};
      \addplot[fill=blue!80,draw=none,opacity=0.5] table {data/experimentB/quant2575res_sk_n900_m3660_s180.dat};
      \addplot[thick,draw=blue] table {data/experimentB/medianres_sk_n900_m3660_s180.dat};
      
      \addplot[fill=red!40,draw=none,opacity=0.5] table {data/experimentB/maxminres_srk_n900_m3660_s180.dat};
      \addplot[fill=red!80,draw=none,opacity=0.5] table {data/experimentB/quant2575res_srk_n900_m3660_s180.dat};
      \addplot[thick,draw=red] table {data/experimentB/medianres_srk_n900_m3660_s180.dat};
    \end{semilogyaxis}
  \end{tikzpicture}
  \begin{tikzpicture}
    \begin{semilogyaxis}[width=0.45\textwidth]
      \addplot[fill=blue!40,draw=none,opacity=0.5] table {data/experimentB/maxminerr_sk_n900_m3660_s180.dat};
      \addplot[fill=blue!80,draw=none,opacity=0.5] table {data/experimentB/quant2575err_sk_n900_m3660_s180.dat};
      \addplot[thick,draw=blue] table {data/experimentB/medianerr_sk_n900_m3660_s180.dat};
      
      \addplot[fill=red!40,draw=none,opacity=0.5] table {data/experimentB/maxminerr_srk_n900_m3660_s180.dat};
      \addplot[fill=red!80,draw=none,opacity=0.5] table {data/experimentB/quant2575err_srk_n900_m3660_s180.dat};
      \addplot[thick,draw=red] table {data/experimentB/medianerr_srk_n900_m3660_s180.dat};
    \end{semilogyaxis}
  \end{tikzpicture}
  \caption{Experiment B: Sparse Kaczmarz (blue) vs. Sparse
    Randomized Kaczmarz (red), $n=900$, $m=3660$, sparsity $s=180$. Left: Plots of relative residual
    $\|Ax-b\|/\|b\|$, right: plots of error
    $\|x-x^\dag\|/\|x^\dag\|$. Thick line shows median over 40
    trials, light area is between min and max, darker area indicate
    25th and 75th quantile.}
  \label{fig:sk-srk-2}
\end{figure}
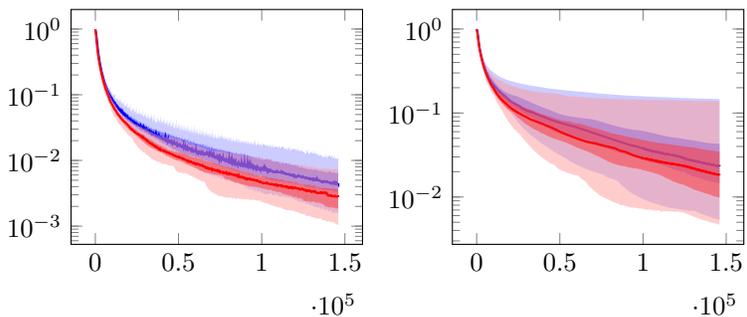

\section{Conclusion}
\label{sec:conclusion}

Using error bounds and the theoretical framework of Bregman projections for split
feasibility problems, we proved expected linear
convergence for the Randomized Sparse Kaczmarz method. Numerical experiments
confirm the linear convergence and demonstrate the benefit of using the method to recover sparse solutions of linear systems, even in the overdetermined case.
However, we could not explicitly quantify the linear rate in terms of the problem data, as for the standard Randomized Kaczmarz method.
The contraction constants $q$ in Theorem~\ref{thm:RBPSFP-L} and Corollary~\ref{cor:lin-conv-sparse-kaczmarz} depend on quantities which are not easily accessible, like the constants $L$ from Theorem~\ref{thm:BregmanBoundcpq} and $\gamma$ from Theorem~\ref{thm:L}.

As demonstrated in~\cite{LSW14} the presented framework also allows for
numerous generalizations which we did not further pursue here.
For example, in the presence of noise we could replace equality constraints $\scp{a_{i}}{x}=b_{i}$ by inequalities $|\scp{a_{i}}{x}-b_{i}|\leq\delta_{i}$ to reflect an error
estimate for each measurement.
Algorithms~\ref{alg:RSK} and~\ref{alg:ERSK} would only have to be changed slightly by projecting onto the modified hyperplanes $H_{\le}(a_i,b_i+\delta_i)$ or $H_{\le}(-a_i,-b_i+\delta_i)$, and we still
obtain linear convergence.

Let us remark that, motivated by the excellent performance of the Randomized Sparse Kaczmarz method, we also tried to solve the regularized nuclear norm problem~\eqref{eq:N} by applying a randomized Kaczmarz iteration of the form~\eqref{eq:nuclear-iter}.
Somewhat disappointingly, our preliminary numerical experiments indicated that this unduly increases the number of times we have to perform the expensive singular value thresholding.
It would be interesting to know if the use of low-rank matrices $A_i$ in~\eqref{eq:nuclear-iter} allows for more efficient updates of $S_\lambda(X_k^{*})$ to compensate for this.
A possible approach could be to use low-rank modifications of the singular value decomposition of the dual iterates $X^*_{k+1} = X^*_{k}-t_k \cdot A_{i}$ as shown in~\cite{Bra06}.

\bibliography{./literature}
\bibliographystyle{plain}

\end{document}